\documentclass{article}
\usepackage[utf8]{inputenc}
\usepackage[hidelinks,plainpages=false]{hyperref}
\usepackage{amssymb,amsthm, amsmath,amstext}
\usepackage{indentfirst}
\usepackage{tikz,tkz-graph}
\usepackage{graphicx}

\GraphInit[vstyle = Empty]

\tikzset{
  VertexStyle/.append style = { inner sep=1pt,
                                font = \bfseries},
  EdgeStyle/.append style = {->,bend left}
 }

\theoremstyle{plain}
\newtheorem{thm}{Theorem}

\newtheorem{definition}{Definition}
\newtheorem{lemma}{Lemma}

\newtheorem*{cor*}{Corollary}
\newtheorem*{conj*}{Conjecture}
\newtheorem{prop}{Proposition}

\title{Peierls barrier for countable Markov shifts}
\author{Jose Chauta\thanks{This study was financed in part by the Coordenação de Aperfeiçoamento de Pessoal de Nível Superior - Brasil (CAPES) - Finance Code 001. Parts of these results were in the author's PhD thesis.} \\
\footnotesize{Department of Mathematics, IME-USP, Brazil}\\
\footnotesize{\texttt{manuelchauta@gmail.com}}
 \and Ricardo Freire\thanks{Supported by FAPESP process 2016/25053-8.} \\
\footnotesize{Department of Mathematics, IME-USP, Brazil}\\
\footnotesize{\texttt{rfreire@usp.br}}
}

\begin{document}
\bibliographystyle{alpha}
\maketitle
\begin{abstract}
We prove the existence of calibrated uniformly continuous subactions for coercive potentials with bounded variation defined on topologically transitive Markov shifts with countable alphabet through the construction of the Peierls barrier in this context. Also, we characterize the existence of bounded calibrated subactions in the same context.
\end{abstract}

\vspace*{5mm}


\vspace*{5mm}


\vspace*{5mm}

\section{Introduction}
Ergodic theory, optimization and thermodynamic formalism have interesting connections. As showed in \cite{blt06,clt01}, some techniques in thermodynamic formalism allow us to find subactions, a useful tool in ergodic optimization. In this paper we extend some results well known in BIP Markov shifts and finite alphabet Markov shifts to noncompact transitive Markov shifts  and coercive potentials with bounded variation.

Given a matrix $A:\mathbb{N}\times \mathbb{N}\to \{0,1\}$ the Markov shift is constructed as the set $\Sigma=\Sigma_A$ of the sequences $x_0x_1\dots $ such that $A(x_i,x_{i+1})=1$, the dynamic $\sigma((x_i)_{i\geq0})=(x_{i+1})_{i\geq0}$ is also called the shift. In this work we assume that $\sigma$ is transitive. Additionally, if a continuous potential $f:\Sigma\to\mathbb{R}$ is defined and the set of $\sigma$-invariant probability measures $\mathcal{M}$ is considered, we can study the next interesting object which is  central in ergodic optimization \cite{Jen}

\begin{equation*}
 m(f):=\sup\left\{ \int f d\mu: \mu\in \mathcal{M}\right\}.
\end{equation*}
 We denote $m(f)$ by $m$ in this paper.
Any measure which attains this supremum is called a maximizing measure.  Existence and properties of maximizing measures over noncompact shift spaces have been studied recently. In \cite{bf14}, \emph{e.g}, it is shown that there exist maximizing measures for coercive potentials with finite variation. It is also proved that there exists a finite subshift which supports every maximizing measure.

Given a continuous potential $V$ satisfying
\begin{equation*}
 f(x)\leq V\circ \sigma(x)-V(x)+m(f)\,,
\end{equation*}
we say that $V$ is a subaction (for $f$). The set of points such that $ f(x)= V\circ \sigma(x)-V(x)+m(f)$ is called the contact set \cite{garibaldi08, garibaldi09}. This set is contained in the set of $f$-nonwandering points. The set of $f$-nonwandering points are the basis on which the Peierls barrier is defined. One important aspect of subactions is that it allows us to characterize the support of all maximizing measures: every maximizing measure has its support in the contact set. Also, if the maximizing measure is unique, then the contact set is uniquely ergodic. A subaction $V$ is calibrated if for any $x\in\Sigma$ there exists $y$ in the contact set such that $\sigma(y)=x$. Calibrated subactions have been deeply studied, see \cite{clt01, garibaldi08}. For example, it has been proved, for the compact and BIP cases, that when the maximizing measure is unique, there is only one calibrated subaction up to a constant.

Given $x,y\in\Sigma$ we define the Peierls barrier $S_f(y,x)$ as a measure of the 'best way to go from $y$ to $x$ maximizing the free energy' (see \cite{Edgardo} and the formal definition in the following section). This construction has been used in  \cite{clt01, Edgardo} and others in order to construct uniformly continuous subactions. As we will show, this construction can be done even without the existence of a Gibbs measure, that is, outside the BIP case \cite{sarig03}: our aim in this paper is to prove that the notion of Peierls barrier, \cite{clt01, garibaldi08} for finite alphabet Markov shift and \cite{Edgardo} in the BIP case, can still be well defined in the general case of transitive Markov shifts. Also, we prove that this barrier, just as in the previous cases, defines a uniformly continuous subaction generalizing some of the mentioned results.

From now on, we denote by $\mu_{max}$ one of the maximizing measures for $f$, and by \cite{bf14} we know it exists since the potential is coercive.

Our first result shows that the Peierls Barrier is well defined when we fix $\overline{y}$ in the $f$-nonwandering set. Also, we prove that the barrier is a bounded above subaction, which is not obvious from its definition.
\begin{thm}\label{ThmA}
 Let $\Sigma$ be a topologically transitive Markov shift and $f$ be a coercive potential with bounded variation. Fix $\overline{y}\in \textrm{supp}(\mu_{max})$, then the Peierls barrier  $S_f:=S_f(\overline{y},\cdotp)$ is a calibrated subaction, which is uniformly continuous and bounded above.
\end{thm}

The following result is a generalization  for some of the results in \cite{clt01,Edgardo}. We prove that the Peierls barrier is an infimum within the set of continuous subactions and, if there exists an unique maximizing measure, they are unique up to a constant.

\begin{thm}\label{ThmB}
Let $\Sigma$ be a topologically transitive Markov shift and $f$ be a coercive potential with bounded variation.  Fix $\overline{y}\in \textrm{supp}(\mu_{max})$ and let $S_f$ be the Peierls barrier $S_f(x)=S_f(\overline{y},x)$. Then
\begin{enumerate}
 \item If $V$ is another continuous subaction,  then for any $x\in\Sigma$ \[S_f(x)\leq
V(x)-V(\overline{y}).\]
\item If $f$ is a H\"older continuous potential, then $S_f$ has bounded variation.
\item If the maximizing measure is unique, then for every bounded calibrated subaction $V$,
\[S_f(x)=
V(x)-V(\overline{y}) \quad \forall x\,,\]
and in particular,  two bounded calibrated subactions differ only by a constant.
\end{enumerate}
\end{thm}

Since $\Sigma$ is noncompact, there is no guarantee that any calibrated  subaction must be bounded. In the next result, we characterize this situation in our general setting, noticing that boundedness is equivalent to $\Sigma$ satisfying one part of the notion of BIP, the BP condition, defined in the next section.

\begin{thm}\label{ThmC}
 Let $\Sigma$ be a topologically transitive Markov shift, $f$ a coercive potential with bounded variation and $V$ a continuous and bounded above subaction. Then, $V$ is bounded if, and only if, $\Sigma$ satisfies the BP condition.
\end{thm}

As a direct consequence of this theorem, when we have uniqueness of the maximizing measure, we have the following.
\begin{cor*}\label{corollary}
Let $\Sigma$ be a topologically transitive Markov shift and $f$ a coercive potential with bounded variation. If the maximizing measure is unique, then the next three statements are equivalent:
\begin{itemize}
 \item[a)] The Peierls barrier is bounded below.
 \item[b)] $\Sigma$ is a BP Markov shift.
 \item[c)] There exists a bounded calibrated subaction.
\end{itemize}
\end{cor*}

This corollary helps us prove in an explicit example, the renewal shift, that there is no bounded calibrated subactions. This construction and some comments appear in the section \ref{renewal}. To the best of our knowledge, this is the first time such an example is explicitly shown.

Our technique follows a progressive restriction of the alphabet constructing a sequence of compact transitive  subshifts. This sort of construction has been used in \cite{mu01} to prove existence of eigenmeasures for the Ruelle operator, in \cite{bf14} to prove existence of maximizing measure and, recently, in \cite{Victor} to prove existence of the equilibrium state for summable potentials with bounded variation in transitive Markov shifts. Such approximation is done by constructing a family of compact subshifts, \emph{i.e.} a restriction  of $\Sigma$ to a finite alphabet.

There are several difficulties in this kind of strategy, as most thermodynamic operators are harder to work with in the whole space and we are left to work with them in the compact subshifts and then deal with limits that are not easily controlled. As an example in this paper, we recall that for the BIP and compact cases that, given $x$, it is well known that any calibrated pre-orbit, \emph{i.e.} a sequence $z_i$ such that $\sigma^i(z_i)=x$ and $z_i$ is in the contact set, is contained in the union of a finite set of cylinders and that any accumulation point of the sequence is in the support of a maximizing measure. In our general case, the calibrated pre-orbit could have large initial letters and no accumulation points, so we will have to approximate $x$ by points in the compact subshifts, where we can repeat the previous constructions, and then take care of the limits.

This document is organized as follows: in the next section we set up the context and make the definitions, in section \ref{proof1} we prove theorem \ref{ThmA}, in section \ref{proof2} we demonstrate theorems \ref{ThmB} and \ref{ThmC} and the corollary. Finally, in the last section we show the example where it is not possible to construct bounded calibrated subactions.

\section{Preliminaries}

 Given a matrix $A:\mathbb{N}\times\mathbb{N}\to \{0,1\}$, which is called the incidence matrix, we define $\Sigma_A$ being the set  of sequences $x=x_0x_1\dots x_n\dots\in\mathbb{N}^{\mathbb{N}}$ such that $A(x_j,x_{j+1})=1$ for any $j\geq0$. The transformation defined on $\Sigma_A$ by $\sigma(x_0x_1\dots x_n\dots)=x_1x_2\dots$ is called the shift. This space has been studied for instance in \cite{sarig99, mu01, sav}. As $A$ is fixed, we write $\Sigma := \Sigma_A$ for convenience.\\
 A cylinder in $\Sigma$ is a set defined by
  \begin{equation*}
  [y_0y_1\dots y_{n-1}]=\{x\in\Sigma: x_0=y_0,x_1=y_1,\dots x_{n-1}=y_{n-1}\}.                                                                                                                                                                                                                                                                                                                                                                                                                                                                                                                                                     \end{equation*}
   This kind of sets are a basis for the topology in $\Sigma$. This is the same topology defined by the distance \[d_\lambda(x,y)=\lambda^{\inf_j\{x_j\neq y_j\}}\]
  where $\lambda\in(0,1)$ is a fixed parameter. If the concatenation, or the word, $ y_0y_1\dots y_{n-1}$ defines a nonempty cylinder in $\Sigma$, it is called an admissible word. The length of the cylinder is defined as the length of the corresponding admissible word. The alphabet, in this case $\mathbb{N}$, is denoted by $\mathcal{A}(\Sigma)$.

  The Markov shift $\Sigma$ is transitive if for any $i,j\in\mathbb{N}$ there exist some $n\geq0$ such that $\sigma^{-n}[j]\cap[i]\neq\emptyset$.  This is equivalent to say that for any $i,j$ there exist an admissible word $v$ such that $ivj$ is admissible.

  \begin{definition}
  The Markov shift is BIP, or satisfies the BIP condition, if it satisfies the following conditions:
  \begin{itemize}
   \item[(BP)] There exist $N$ such that for any $j\in\mathbb{N}$ there exist $i\leq N$ such that $ij$ is admissible.
   \item[(BI)] There exist $N$ such that for any $j\in\mathbb{N}$ there exist $i\leq N$ such that $ji$ is admissible.
  \end{itemize}
  \end{definition}
  It is well known that for the topologically mixing case, see \cite{sarig03}, that $\Sigma$ satisfies the BIP condition if, and only if, there exists Gibbs measures.

  Given a bounded above potential $f:\Sigma_A\to\mathbb{R}$, we say that $f$ has bounded variation if \[Var(f)=\sum_{j=1}^{\infty}Var_j(f)<\infty,\] where $Var_j(f)=\sup_{x,y}\{|f(x)-f(y)|: d(x,y)\leq \lambda^j\}$. If a potential has bounded variation, then it is uniformly continuous. The potential $f$ is  H\"older continuous if there exist $K>0$ and $r\in(0,1)$ such that $V_n(f)\leq Kr^n$ for any $n\geq1$.

  The potential   $f$ is coercive if
  \[\lim_{j\to\infty}\sup\{f(x): x\in[j]\}=-\infty.\]
We notice that any continuous and coercive potential is bounded above.

Given a bounded above potential  $f$, $x\in\Sigma$ and $n\geq0$ we define $S_nf(x)=\sum_{j=0}^{n-1}f(\sigma^j(x))$. A periodic point $x$ with period $n$ defines the natural invariant measure $\mu_x$: \[\mu_x := \frac{1}{n}\sum_{j=0}^{n-1}\delta_{\sigma^j(x)} \,,\] where $\delta_{\sigma^j(x)}$ is the Dirac measure on $\sigma^j(x)$.

 Let $\Sigma$ be a transitive Markov shift and $f$ a bounded above continuous potential, we define
 \begin{equation*}
  m(f)=\sup\left\{\int f d \mu: \ \mu \text{ is an invariant probability measure.} \right\}
 \end{equation*}

The study of $m(f)$ and its properties has been named \emph{Ergodic Optimization}, see for example \cite{Jen}. A probability measure which attains that supreme is called a \emph{maximizing measure}. As we have mentioned, existence of maximizing measures for coercive potential with bounded variation in transitive Markov shifts has been proved in \cite{bf14}. In fact, there exists a compact subshift such that any maximizing measure has to be supported in that subshift.

One of the most useful tools in ergodic optimization is the subaction, which has been studied in detail in several cases. See, for example \cite{clt01, garibaldi08} or \cite{garibaldiBook}.

 Let $f$ be a potential on a transitive Markov shift $\Sigma$. Then $V: \Sigma\to\mathbb{R}$ is a subaction for $f$ if
 \begin{equation*}
  f(x)\leq V\circ\sigma(x)-V(x)+m
 \end{equation*}
for all $x$ in $\Sigma$. In addition, if given $x\in\Sigma$ there exists $y$ such that $\sigma(y)=x$ and $f(y)=V(x)-V(y)+m$, then $V$ is called a {\bf calibrated subaction}.

Given a finite subset $B\subset \mathbb{N}$, the subshift $\Sigma_B$ is the compact shift given by $x\in\Sigma$ such that $x_i\in B$ for any $i\geq0$. The dynamic $\sigma$ can be restricted to this subshift. We stress that a subshift might not be transitive for any choice of $B$. Although, an increasing sequence of transitive compact \emph{subshifts}, $(\Sigma_k)_{k\geq0}$ that  satisfies $\Sigma_{k-1}\subset \Sigma_k $  for any $k\geq1$ and any nonnegative integer is in the alphabet of some $\Sigma_j$ can be constructed. For example, this kind of construction has been used in \cite{mu01} to prove existence of eigenmeasures for the Ruelle operator, in \cite{bf14} to prove existence of maximizing measure and, recently, in \cite{Victor} to prove existence of the equilibrium state for summable potentials with bounded variation in transitive Markov shifts. Also, as in \cite{Victor, bf14} it can be done such that for a strictly increasing sequence $n_k$ the alphabet $\mathcal{A}_k$ of the subshift $\Sigma_k$ contains every letter smaller or equal to $n_k$.  By \cite{bf14}, there is no loss of generality in assuming that every maximizing measure is supported in $\Sigma_0$. For any $k\geq0$ let $f_k$ be the restriction of $f$ to $\Sigma_k$, we take notice that $m(f_k)=m(f)$ for every $k\geq0$.

An element $y\in\Sigma$ is $f$-nonwandering if for all $\epsilon>0$ there exist $z\in\Sigma$ and a non negative integer $n$ such that $d(y,z)<\epsilon$, $\sigma^n(z)=y$ and
\begin{equation*}
 |S_n(f-m)(z)|<\epsilon.
\end{equation*}
The set of $f$-nonwandering points is denoted by $\Omega(f)$.

The existence of $f$-nonwandering points for the finite alphabet case  was proved in \cite{clt01}. Also,  it is  proved that every $f$ maximizing measure has its support contained in $\Omega(f)$. As a consequence, the $f$-nonwandering set is non-empty.

Now we introduce the Peierls barrier. We follow the notation and ideas of \cite{clt01,garibaldi08} and, particularly, in \cite{Edgardo} for a similar construction in the BIP case.

\begin{definition}
 Let $x,y\in\Sigma$ and $\epsilon>0$, define
 \begin{equation}\label{Sf}
  S_f^\epsilon(y,x):=\sup_{n\geq0}\sup_{z}\{S_n(f-m)(z):\ d(y,z)<\epsilon,\ \sigma^n(z)=x \}
 \end{equation}
and
\begin{equation*}
 S_f(y,x):=\lim_{\epsilon\to0}S_f^\epsilon(y,x).
\end{equation*}
\end{definition}
This construction can be done for every $x,y$ in $\Sigma$ although the most interesting case appears when $y$ is a $f$-nonwandering point. For example when it belongs to the support of any maximizing measure. As it is showed in \cite{clt01}, that condition is sufficient in the finite alphabet case to show that $S_f(y,x)>-\infty$. This observation will be useful in our proof of the same result for the noncompact case.

As in the finite alphabet case, $S_f$ can be defined with $\limsup_{n\to\infty}$ instead of $\sup_n$ in \eqref{Sf}. The proof of this assertion can be found in \cite{clt01} for the Holder continuous case, a similar argument can be used for our setting.


\begin{prop}
Let $\Sigma$ be a transitive Markov shift and $f$ be a potential with bounded variation. Given $y\in\Omega(f)$ we have
\begin{equation*}
  S_f(y,x)=\lim_{\epsilon\to0}\limsup_{n\to\infty}\sup_z\{S_n(f-m)(z)|\sigma^n(z)=x,\ d(z,y)<\epsilon\}
\end{equation*}
\end{prop}
\proof
Given $\epsilon\in ]0,\lambda[$ and $z$ such that $\sigma^n(z)=x$, $d(z,y)<\epsilon$ and $S_n(f-m)(z)\geq S_f^\epsilon(x)-\epsilon$. For $\frac{\epsilon}{3}$ there exist $z'$ such that $d(z',y)<\frac{\epsilon}{3}$, $\sigma^k(z')=y$ and $|S_n(f-m)(z')|<\epsilon/3$.
From $d(z,y)<\lambda$ we have $z_0=y_0$ so $\tilde{z}=z'_0z'_1z'_2\dots z'_{k-1}z$ is in $\Sigma$ and $d(z',\tilde{z})<\lambda^{k+L}$ where $L$ satisfies $\lambda^L\leq\epsilon$. From its definition $\sigma^{k+n}(\tilde{z})=x$ and
\begin{eqnarray*}
  S_{k+n}(f-m)(\tilde{z}) &=& S_k(f-m)(\tilde{z})+S_n(f-m)(z) \\
   &\geq&  S_k(f-m)(z')-\sum_{j=1}^{k}Var_{j+L}(f)+S_n(f-m)(z)  \\
   &\geq& -\frac{\epsilon}{3}-\sum_{j=1}^{k}Var_{j+L}(f)+(S_f^\epsilon)(x)-\epsilon
\end{eqnarray*}

and $k$ can be chosen as large as we want, then for the given $\epsilon$ we have \[\limsup_{n\to\infty}\sup_z\{S_n(f-m)(z)|\sigma^n(z)=x,\ d(z,y)<\epsilon\}\geq S_f^\epsilon-\frac{4}{3}\epsilon +\sum_{j=L}^{\infty}Var_{j}(f)\, .\]
Taking $\epsilon\to 0$ we conclude the proof, observe that $L$ depends on $\epsilon$ and $f$ have bounded variation.

\endproof


\section{Proof of the theorem \ref{ThmA}}\label{proof1}
We first  outline the proof of the main result of the paper. Initially we prove that $S_f(y,x)$ is well defined when $y$ is a non-wandering point, secondly, we prove that when restricted to a compact subshift the Peierls barrier on the entire space coincide with the barrier defined on some subshift that contains the former compact subshift (lemmas \ref{cut1} and \ref{cut2}). To conclude the proof of the theorem \ref{ThmA}, we use the previous step and results of the compact case \cite{clt01, garibaldi08}.

Now we prove that Peierls barrier is well defined.
\begin{lemma}
 Let $\Sigma$ be a transitive Markov shift and $f$ be a coercive potential with bounded variation. Let $y$ be in the support of any maximizing measure. Then $S_f(x):=S_f(y,x)$ satisfies $-\infty<S_f(x)<\infty$.
\end{lemma}
\proof

Let us fix $x\in\Sigma$ and $0<\epsilon<\lambda$. To prove that $S_f(y,x)<\infty$ is sufficient  to show that $S_f^\epsilon(y,x)<\infty$, because it is decreasing as $\epsilon$ goes to $0$. Any $z$ such that $d(y,z)<\epsilon$ and $\sigma^n(z)=x$ is in the form
\begin{equation*}
 z=y_0y_1\dots y_lw_{l+1}\dots w_{n-1}x
\end{equation*}
where $\lambda^{l+1}\leq\epsilon<\lambda^{l}$. By the transitivity of $\Sigma$ there exists an admissible word $v=v(x_0,y_0)$ which connects $x_0$ with $y_0$, it means that $x_0vy_0$ is an admissible word.

Let us define a periodic point  $\bar{z}\in\Sigma$ by \[\bar{z}=y_0y_1\dots y_lw_{l+1}\dots w_{n-1}x_0vy_0y_1\dots \]
and denote by $P$ the period of $\bar{z}$, $P=n+|v|$. then
\begin{equation*}
 S_Pf(\bar{z})\leq Pm\,,
\end{equation*}
and if we write it in a different way, we get
\begin{equation}\label{A}
 S_nf(\bar{z})-nm\leq |v|m-S_{|v|}f(\sigma^n(\bar{z})).
\end{equation}

Define
\begin{equation*}
 \alpha:=\inf_{n\leq j\leq P}\inf_{x\in[z_j] }\{f(x)\}.
\end{equation*}

Then \begin{equation*}
      S_{|v|}f(\sigma^n(\bar{z}))\geq |v|\alpha.
     \end{equation*}
From this and \eqref{A} we get
\begin{equation}\label{B}
 S_nf(\bar{z})-nm\leq |v|(m-\alpha).
\end{equation}

On the other hand, $d(z,\bar{z})<\lambda^{n+1}$, and since  $f$has bounded variation, see \cite{Mengue, Edgardo} and references therein,
\begin{equation}\label{C}
S_n(f-m)(z)-S_n(f-m)(\bar{z})\leq\sum_{j=0}^{n}Var_j(f)\leq Var(f).
\end{equation}

Finally, from \eqref{B} and \eqref{C}
\[S_{n}f(z)-nm\leq S_{n}f(\bar{z})-nm+Var(f) \ \leq |v|(m-\alpha) + Var(f).\]

Taking $\sup_n\sup_z$
\begin{equation*}
 S_f^\epsilon(x) \leq |v|(m-\alpha) + Var(f)<\infty.
\end{equation*}

Note that $\alpha$ and $|v|$ depend only on $x_0$ and $y_0$, that observation will be useful when we restrict the Peierls barrier to compact subshifts.

We introduce some notation for the Peierls barrier in the sequence $\left(\Sigma_k\right)_{k\geq0}$ of compacts subshifts:
\[S_{f_k}(x)=\lim_{\epsilon\to0}\sup_n\sup_z\{S_n(f-m)(z)|\ \sigma^n(z)=x,\ d(z,y)<\epsilon, \text{ and }
z\in\Sigma_k\}.\]

Now we prove that if $y$ is chosen in the support of a maximizing measure, then $S_f(y,x)>-\infty$ for all $x\in \Sigma$. Given $x$ in a compact subshift $\Sigma_k$, it follows from \cite{clt01, garibaldi08} and from the fact that $y$ is a non-wandering point that $S_{f_k}(y,x)>-\infty$ and so is $S_f(y,x)$.

In the general case, given $x\in\Sigma$ let be $\tilde{x}$ in some compact $\Sigma_k$, such that $d(x,\tilde{x})<\lambda^2$. For $\epsilon>0$ and $z\in\Sigma$ in the form

\begin{equation*}
 z=y_0y_1\dots y_l w_{l+1}\dots w_n x \,,
\end{equation*}
 let us define
\begin{equation*}
 \tilde{z}=y_0y_1\dots y_l w_{l+1}\dots w_n \tilde{x}.
\end{equation*}
From the bounded variation of $f$ we have that $S_nf(z)\geq S_nf(\tilde{z})-Var(f)$, then $S_f(y,x)\geq S_f(y, \tilde{x})-Var(f)$. By the first step we have $S_f(y,\tilde{x})>-\infty$, so $S_f(y,x)>-\infty$.
\endproof

From now on we fix $\overline{y}$ in the support of a maximizing measure. As we mentioned before $\overline{y}\in\Omega(f)$. In order to simplify the notation we denote $S_f(\overline{y},x)$ by $S_f(x)$. The following two lemmas allow us to show that the Peierls barrier restricted to a compact subshift matches the Peierls barrier of some compact subshift that contains the former one.

\begin{lemma}\label{cut1}
 Let $\Sigma$ be a transitive Markov shift and $f$ be a coercive potential with bounded variation. Then for any $a\in\mathbb{N}$ there exists $J=J(a)$ such that, for any $x\in[a]$
\begin{equation*}
 S_f^\epsilon(\overline{y},x)  =  \sup_{n}\sup_{z}\{S_n(f-m)(z): d(\overline{y},z)<\epsilon,\ \sigma^n(z)=x, \text{ and } \exists i \in ]l,n[\text{ s.t. }  z_i<J\}
\end{equation*}
where $l$ is such that $\lambda^l<\epsilon\leq \lambda^{l-1}$.
\end{lemma}

\proof
Without loss of generality we can assume that $m=m(f)=0$. Let $a\in\mathbb{N}$ and $\Sigma_k$ a compact \textit{subshift} such that  $\mathcal{A}_k$ contains all letters less or equal to $a$. Note that $y\in \Sigma_{k_0}$.  As $\Sigma_k$
is transitive, if $b\in\mathcal{A}_k$, there exists $w^a_b$ an admissible word in $\Sigma_k$ which connects  $b$ to $a$. Let us fix
$W=\{w_i^a| i\in\mathcal{A}\}$, a finite set of needed words for make those connections, for each $i$ it is chosen only one connecting word. We define $L_w$  as the biggest length  of words in $W$.

Consider  $J$ such that, for any  $j>J$ and any  $i\in\mathcal{A}_k$
\begin{equation}\label{condicao}
\sup f|_{[j]}<L_w(\inf f|_{[i]})-Var(f).
\end{equation}
This can be done because $f$ is coercive. Also, we can suppose that $\sup f|_{[j]}<0$ for all $j>J$.

Given $\epsilon>0$ and $x\in[a]$, if $\sigma^n(z)=x$ and $d(z,\overline{y})<\epsilon$, then $z$ has the form
$z=y_0y_1\dots y_lw_{l+1}\dots w_{n-1}x$.

Suppose that $w_j>J$ for $j\in\{l+1,l+2,\dots n-1\}$.  Define $\tilde{z}=y_0y_1\dots
y_lw_{y_l}^{a}x$, and observe that $\sigma^{\tilde{n}}(\tilde{z})=x$ where
$\tilde{n}=l+|w_{y_l}^{a}|$, and clearly $d(\tilde{z},y)<\epsilon$.

\begin{equation*}
\begin{split}
S_nf(z) & = \sum_{j=0}^{n-1}f(\sigma^j(z)) \\
        & = \sum_{j=0}^{l}f(\sigma^j(z))+\sum_{j=l+1}^{n-1}f(\sigma^j(z)) \\
        & \leq \sum_{j=0}^{l}f(\sigma^j(\tilde{z}))+Var(f)+\sum_{j=l+1}^{n-1}\sup f|_{[w_j]}, \\
\end{split}
\end{equation*}
from \eqref{condicao} we have
\begin{equation*}
\begin{split}
        S_nf(z) &\leq \sum_{j=0}^{l}f(\sigma^j(\tilde{z}))+Var(f)+\sup f|_{[w_{l+1}]} \\
        & \leq \sum_{j=0}^{l}f(\sigma^j(\tilde{z})+Var(f)+(\inf_j\inf f|_{[\tilde{z}_{j}]})L_w-Var(f)\\
        & \leq S_{\tilde{n}}f(\tilde{z}).
\end{split}
\end{equation*}

Then, for any $z$ such that $d(z,\overline{y})<\epsilon$  and $\sigma^n(z)=x$ there exists $\tilde{z}$ which satisfies
$d(\tilde{z},\overline{y})<\epsilon$, $\sigma^{\tilde{n}}(\tilde{z})=x$ and $S_{\tilde{n}}f(\tilde{z})\geq S_nf(z)$. Notice that $\tilde{z}$  also satisfies $\tilde{z}_j<J$ for some $j\in ]l,n[$.

\endproof

\begin{lemma}\label{cut2}
Let $\Sigma$ be a transitive Markov shift and $f$ be a coercive potential with bounded variation. Then for any $a\in\mathbb{N}$ there exists $N=N(a)$ such that, for any $x\in[a]$
\begin{equation*}
S^\epsilon_f(x)= \sup_{n}\sup_{z}\{S_n(f-m)(z):d(\overline{y},z)<\epsilon,\  \sigma^n(z)=x, \text{and }\ z_i<N\  \forall i<n\}
\end{equation*}
In particular, $S_f(x)=S_{f_J}(x)$, for any $x\in[a]\cap\Sigma_k$ where $J$ is such that $\{1,2\dots N(a)\}\subset\mathcal{A}_J$.
\end{lemma}

\proof
Without loss of generality we can assume that $m=m(f)=0$. Let $a\in\mathbb{N}$, from lemma \ref{cut1} there exist $J=J(a)$ such that for any $x\in[a]$ and $\epsilon>0$,
$S_f^\epsilon (x)$ is defined by

 \begin{equation*}
  S^\epsilon_f(x)= \sup_{n}\sup_{z}\{S_nf(z):d(\overline{y},z)<\epsilon,\  \sigma^n(z)=x, \text{ with } z_i<J \text{ for some } i \in ]l,n[\}
 \end{equation*}
Consider $\Sigma_{k_2}$ a compact subshift such that $\mathcal{A}_{k_2}$ contains every letter less or equal to $J+1$.
$\Sigma_{k_0}\subset\Sigma_k\subset\Sigma_{k_2}$ and $\{0,1,2,3\dots J\}\subset \mathcal{A}_{k_2}$. Since $f$ is a coercive potential,
there exists $I$ such that  for any $i\in\mathcal{A}_{k_2}$ and $j>I$
\begin{equation}\label{J2}
 \sup f|_{[j]}<L_2\inf f|_{[i]}-Var(f).
\end{equation}
$L_2$ is defined by the next construction. Given $i_1,i_2\in\mathcal{A}_{k_2}$ define $w_{i_1}^{i_2}$ some connecting word in  $\Sigma_{k_2}$ from $i_1$ to $i_2$.  Define the finite set of those connecting words $W_2=\{w_{i_1}^{i_2}| \ i_1,i_2\in\mathcal{A}_{k_2}\}$. Denote by $L_2$ the biggest length of words in $W_2$.

If $x\in[a]$ and $\epsilon>0$.  Given $z\in\Sigma$ such that $d(z,\overline{y})<\epsilon$, $\sigma^n(z)=x$, and $z_k>I$ for some  $k<n$, we can find $i_I<k$, $i_D>k$ with $i_I,i_D\in \mathcal{A}_{k_2}$ and $z_i\notin \mathcal{A}_{k_2}$ for $i_I<i<i_D$. This implies that $z$ can be written in the form
\begin{equation*}
 z=y_0\dots y_lz_{l+1}\dots
z_{i_I-1}z_{i_I}z_{i_I+1}\dots z_k\dots
z_{i_{D-1}}z_{i_D}z_{i_D+1}\dots z_{n-1}x.
\end{equation*}

Define, by substitution of the connecting word between $z_{i_I}$ and $z_{i_D}$,
\begin{equation*}
 \tilde{z}=y_0\dots y_lz_{l+1}\dots
z_{i_I-1}z_{i_I}w_{i_I}^{i_D}z_{i_D}z_{i_D+1}\dots z_{n-1}x.
\end{equation*}

Then $\tilde{z}$ satisfies $d(\tilde{z},\overline{y})<\epsilon$, $\sigma^{\tilde{n}}(\tilde{z})=x$. Where $\tilde{n}=i_I+|w_{i_I}^{i_D}|+(n-i_D)$.
Our next aim is to prove that $S_nf(z)\leq S_{\tilde{n}}f(\tilde{z})$.
\begin{equation*}
\begin{split}
S_nf(z) & = \sum_{j=0}^{n-1}f(\sigma^j(z)) \\
        & = \sum_{j=0}^{i_I}f(\sigma^j(z))+\sum_{j=i_I+1}^{i_D-1}f(\sigma^j(z)))+\sum_{j=i_D}^{n-1}f(\sigma^j(z))).
\end{split}
\end{equation*}
In the first term we have $d(z,\tilde{z})<\lambda^{i_I}$, then
\begin{equation}\label{1}
 \sum_{j=0}^{i_I}f(\sigma^j(z))\leq \sum_{j=0}^{i_I}f(\sigma^j(\tilde{z}))+Var(f).
\end{equation}
For the second term,
\begin{equation*}
 \sum_{j=i_I+1}^{i_D-1}f(\sigma^j(z)))\leq f(\sigma^k(z)))
\end{equation*}
therefore, by \eqref{J2},
\begin{equation*}
f(\sigma^k(z))\leq L_2(\inf_{i\in\mathcal{A}_{k_2}}\inf_{y\in[i]} f(y))-Var(f).
\end{equation*}

Define $C=|w_{i_I}^{i_D}|$, then
\begin{equation*}
 L_2\left(\inf_j\inf_{y\in[\tilde{z}_{j}]}f(y)\right)\leq \sum_{j=i_I}^{i_I+C}f(\sigma^j(\tilde{z})),
\end{equation*}
We obtain
\begin{equation}\label{2}
  \sum_{j=i_I+1}^{i_D-1}f(\sigma^j(z)))\leq\sum_{j=i_I}^{i_I+C}f(\sigma^j(\tilde{z})).
\end{equation}
Now observe that $\sigma^{i_D}(z)=\sigma^{i_I+C}(\tilde{z})$, then

\begin{equation}\label{3}
 \sum_{j=i_D}^{n-1}f(\sigma^j(z)))=\sum_ {i_I+C}^{\tilde{n}-1}f(\sigma^j(\tilde{z})).
\end{equation}

From \eqref{1}, \eqref{2} and \eqref{3}, we obtain
\begin{equation*}
S_nf(x)\leq  S_{\tilde{n}}f(\tilde{z}).
\end{equation*}

 The same argument can be repeated finite times in order to obtain $z'\in\Sigma$ such that $\sigma^{n'}(z')=x$,
$d(z',\overline{y})<\epsilon$, $z'_i<J_2$ for any $j<n$ and  $S_{n'}f(z')\geq S_nf(z)$.

\endproof

As a direct consequence of lemma  \ref{cut2} he have

\begin{prop}\label{compacto}
 For  every $J\in \mathbb{N}$, there exists $K\geq J$ such that $S_f|_{\Sigma_J}=S_{f_K}$. In addition, if  $x\in\Sigma_J$, there exists $y\in\Sigma_K$ such that $\sigma(y)=x$ and \[S_f(y)=S_f(x)-f(y)+m.\]
\end{prop}

\proof
Let us consider $I(i)$ an integer from lemma \ref{cut2}, and define $I=\max\{I(j),\ j\in\mathcal{A}_k\}$, to complete the first part of the proof,  consider $\Sigma_J$ such that $i\in\mathcal{A}_J$ for any $i\leq I$.

To prove the second part, for any $x\in\Sigma_J$ consider $y\in\Sigma_K$ such that $\sigma(y)=x$ and $S_{f_J}(y)=S_{f_J}(x)-f(y)+m$, for such $y$ we have
\begin{equation*}
 S_f(y)=S_f(x)-f(y)+m.
\end{equation*}

\endproof
This proposition implies that $S_f$ is calibrated on any compact subshift. Now we prove the same result in our setting.
\begin{lemma}\label{calibratedlemma}
 Let $\Sigma$ be a transitive Markov shift and $f$ be a coercive potential with bounded variation. Then the Peierls barrier is a calibrated, uniformly continuous subaction.
\end{lemma}

\proof
As we have already proved that $S_f$ is a well defined function, we show, based on the argument for the compact case, that it is a uniformly continuous function and by using the previous lemmas we prove the calibrated part of the result.

Let $x\in\Sigma$ and $y\in\sigma^{-1}(x)$, for any $\epsilon>0$ and $n\geq2$

\begin{align*}
\sup\{S_n(f-m)(z):\ & \sigma^n(z)=x, d(z,y)<\epsilon\}  \geq    \\
    &\sup\{S_{n-1}(f-m)(z):\ \sigma^{n-1}(z)=y, d(z,y)<\epsilon\}+f(y)-m \,.
\end{align*}
Taking $\sup$ in $n$ and then the limit in $\epsilon$ we obtain $S_f(x)\geq S_f(y)+f(y)-m$, so $S_f$ is a subaction.

As we mentioned before, $S_f$ is calibrated for points in compact subshifts. In the general case, we use lemma \ref{cut2}. Consider $x\in\Sigma$ and construct a sequence $z^k$ such that, $\sigma^{n_k}(z^k)=x$ for some sequence $n_k$ and $\lim_{k\to\infty}S_{n_k}(f-m)(z^k)=S_f(x)$. We can assume that $z^k_j<N(x_0)$ for every $j\leq n_k$, see lemma \ref{cut2}. Then there exist a subsequence $n_{k_j}$ such that $z^{k_j}_{n_{k_j}}=L$ for a fixed $L\leq N(x_0)$. To simplify the notation we write  $k_j$ by $j$.

Define $y:=Lx_0x_1\dots$, observe that $\sigma^{n_j-1}(z^j)=y$, then
\[S_f(y)\geq \lim_{j\to\infty}S_{n_j-1}(f-m)(z^j) \,,\]
and adding $f(y)-m$ to this inequality we obtain
\[S_f(y)+f(y)-m\geq \lim_{j\to\infty}S_{n_j}(f-m)(z^j)=S_f(x)\,.\]

On the other hand, as $S_f$ is a subaction we get the opposite inequality. Then $S_f(y)+f(y)-m=S_f(x)$, then $S_f$ is calibrated.

On the regularity of $S_f$, we prove that, as in the compact case, $V_n(S_f)\leq \sum_{j=n}^{\infty}V_j(f)$ and by the fact that $f$ has bounded variation, $V_n(S_f)\to0$ as $n\to\infty$.

Consider $x,w\in\Sigma$ such that $d(x,w)<\lambda^n$ and $\epsilon>0$, then for any $z_x$ such that $d(\overline{y},z_x)<\epsilon$ and $\sigma^m(z_x)=x$ define $z_w$ such that  $d(z_x,z_w)<\lambda^{n+m}$ and $\sigma^m(z_w)=w$, then $d(z_w,\overline{y})< \epsilon$,  and
\begin{equation}
S_m(f)(z_w)\leq S_mf(z_x)+\sum_{j\geq n}Var_j(f)
                                                                                                                                                                                     \end{equation}
changing the order between $x$ and $w$ and taking supreme in $n$ and $z$ we get
\begin{equation*}
|S_f^\epsilon(w)-S_f^\epsilon(x)|<\sum_{j\geq n}Var_j(f)                                                                                                                                                                                                                                                                                                                                            \end{equation*}
This inequality holds for any $\epsilon>0$, so  $Var_n(S_f)\leq \sum_{j\geq n}Var_j(f)$. Obviously, this implies that $S_f$ is uniformly continuous.
\endproof

Finally, to complete the proof of theorem A we have the following.

\begin{lemma}
Let be $\Sigma$ a transitive Markov shift and $f$ be a coercive potential with bounded variation, then $S_f$ is a bounded above subaction.
\end{lemma}

\proof

We fix $\lambda>\epsilon>0$ and we consider $J\in\mathbb{N}$ such that $\sup f|_{[i]}<-Var(f)$ for $i>J$. Let us consider $\Sigma_k$ such that $\{1,2,3,...,J\}\subset\mathcal{A}_k$ and for each $j\leq J$ a point $x^j\in\Sigma_k\cap\sigma([j])$, those points exist because $\Sigma_k$ is transitive.\\
Given $x\in\Sigma$ for any $z\in\Sigma$ such that $\sigma^n(z)=x$ and $d(z,\overline{y})<\epsilon$ we have two options:
\begin{itemize}
 \item $z_{n-1}=j\leq J$ or
 \item $z_{n-1}>J$.
\end{itemize}
For the first case we define $\tilde{z}=z_0z_1\dots z_{n-1}x^j$ we observe $d(z,\tilde{z})\leq \lambda^{n-1}$ so $S_nf(z)\leq S_nf(\tilde{z})+Var(f)$.
Also we observe that $\tilde{z}$ satisfies $\sigma^n(\tilde{z})=x^j$, $d(\tilde{z},\overline{y})<\epsilon$, then $S_nf(\tilde{z})\leq S_f^\epsilon(x^j)$. Then
\begin{equation}\label{delta1}
 S_nf(z)\leq S_f^\epsilon(x^j)+Var(f).
\end{equation}

In the second case, let $j\in]l,n-1[$ the maximum integer such that $z_j\leq J$ where $\lambda^l\leq\epsilon$. Then $S_nf(z)=S_{j}(z)+S_{n-j}(\sigma^{j}(z))$ and notice that $S_{n-j}(\sigma^{j}(z))<-(n-j)Var(f)$, and also $\sigma^{j}(z)$ is in the first case. Then
\begin{equation}\label{delta2}
S_nf(z)< S_f^\epsilon(x^{z_{j+1}}) +Var(f)-(n-j)Var(f).                                                                                                                                                                         \end{equation}

If for all $j\in]l,n[$ we have $z_j>J$ it is clear that
\begin{equation}\label{delta3}
 S_nf(z)<\sup_lS_lf(y)+Var(f)-(n-l)Var(f).
\end{equation}
From \ref{delta1}, \ref{delta2}, \ref{delta3} and by defining $\Delta_1=\max\{S_f^\epsilon(x^j): j\leq J\}$ and $\Delta_2=\sup\{S_lf(y):l\in\mathbb{N}\}$. We have
\[S_nf(z)\leq\max\{\Delta_1,\Delta_2\}\]
This implies that $S_f^\epsilon(x)\leq \max\{\Delta_1,\Delta_2\}$ for all $x\in\Sigma$

\endproof

\section{Proof of theorems \ref{ThmB} and \ref{ThmC}}\label{proof2}
In this section we discuss some quotes on the Peierls barrier. Initially we prove theorem B.

Recall that a calibrated pre-orbit is a sequence $(x_k)$ such that $\sigma^k(x_k)=x$ and $S_f(x_k)=S_f(x_{k-1})+f(x_k)-m$. When the alphabet is finite, it is shown in \cite{clt01} that every accumulation point of a calibrated pre-orbit belongs to the support of a maximizing measure. By using lemma \ref{cut2} we prove the analogous result in our context. Currently we cannot  control the possibility that a calibrated sequence have a non-bounded initial letter. Also we point out that by lemma \ref{calibratedlemma}, every $x\in\Sigma$ has a calibrated pre-orbit.

 \begin{lemma}
 Let be $\Sigma$ a transitive Markov shift and $f$ be a coercive potential with bounded variation. Given $x\in\Sigma_k$ for some $k\geq0$, then there exists a calibrated pre-orbit contained in a compact subshift. Any accumulation point of such a calibrated pre-orbit belongs to the support of a maximizing measure.

 \end{lemma}

\proof
Given $x\in\Sigma_k$ for some $k\in\mathbb{N}$, let be $\Sigma_J$ such that $S_f|_{\Sigma_k}=S_{f_J}|_{\Sigma_k}$. For that restriction there exists $y\in\Sigma_J$ such that $S_{f_J}(y)=S_{f_J}(x)-f(y)+m=S_f(x)-f(y)+m\geq S_{f}(y)$. So $S_{f_J}(y)=S_f(y)$, \emph{i.e.} $y$ satisfies
\begin{equation*}
 S_{f}(y)=S_f(x)-f(y)+m.
\end{equation*}
Note that  $y\in\Sigma_J$ and $S_{f_J}$ is a calibrated subaction for $f_J$. This implies that there exist $y^2\in\Sigma_J$ such that $\sigma(y^2)=y$ and
\begin{equation*}
 S_{f_J}(y^2)=S_{f_J}(y)-f(y^2)+m=S_{f}(y)-f(y^2)+m\geq S_f(y^2).
\end{equation*}
On the other hand, $S_f\geq S_{f_J}$, so $S_f(y^2)=S_{f_J}(y^2)$ and
\begin{equation*}
 S_{f}(y^2)=S_{f}(y)-f(y^2)+m.
\end{equation*}

This argument can be applied for $y^2$ to find a $y^3$ also in $\Sigma_J$. Recursively, we can construct a calibrated pre-orbit contained in $\Sigma_J$.
To show that any accumulation point of the sequence $y^k$ belongs to a maximizing measure we observe that this is true for $\Sigma_J$, as in \cite{Edgardo}. Notice that any maximizing measure for $f_J$ is a maximizing measure for $f$.

\endproof

Next, we prove theorem B. We emphasize that in this setting, it might not exist a bounded calibrated subaction, see section \ref{renewal}

\proof[Proof of Theorem \ref{ThmB}]
Let  $V$ be a continuous subaction and  $x\in\Sigma$. For each $y\in\Sigma$ such that $\sigma(y)=x$, $V(y)\leq V(x)-f(y)+m$. In the same way, if $\sigma(y^2)=y$,
\begin{equation*}
 \begin{split}
  V(y^2)&\leq  V(y)-f(y^2)+m\\ &\leq (
V(x)-f(y)+m)-f(y^2)+m\\&=V(x)-\sum_{j=0}^{1}(f(y^j)-m).
 \end{split}
\end{equation*}
This implies that for any  $y\in\Sigma$ satisfying $\sigma^n(y)=x$, we have

\begin{equation}\label{seq}
  V(y)\leq V(x)-\sum_{j=0}^{n-1}f(\sigma^j(y))-m.
\end{equation}
By definition of $S_f$ there exists a sequence  $(y^k)_{k\in\mathbb{N}}$ such that $\sigma^{n_k}(y^k)=x$, $d(y^k,\overline{y})\to0$,
 and $\lim_{k\to\infty}S_{n_k}(f-m)(y^{k})=S_f(x)$, because of \eqref{seq} we have

\begin{equation*}
  \lim_{k\to\infty}V(y^k)\leq \lim_{k\to\infty}V(x)-\sum_{j=0}^{n_k-1}f(\sigma^j(y^k))-m.
\end{equation*}
On the left hand we obtain $V(\overline{y})$, by continuity of $V$. So

\begin{equation*}
  V(\overline{y})\leq V(x)-S_f(x).
\end{equation*}
This proves the first part of the theorem.

 For the second part of the theorem, as we have shown,  $Var_l(S_f)\leq \sum_{j=l}^{\infty}Var_j(f)$. If $f$ is  H\"older continuous,
$Var_j(f)\leq A\theta^j$. Then  \[Var(S_f)= \sum_{l\geq1}Var_l(S_f)\leq
\sum_{l\geq1}\sum_{j=0}^{l-1}Var_j(f)=\sum_{l\geq1}\sum_{j=0}^{l-1}A\theta^j.\]
That last series is convergent because it can be written in the form
\[\sum_{j=0}^{\infty}Aj\theta^j.\]
This proves the second part of the theorem.

  For the last part of the theorem we use the known result in compact shifts  \cite{blt06} and an approximation argument. This argument uses the density of the union of compact sub-shifts in the whole space.

\begin{prop}
In the hypothesis of theorem \ref{ThmB}, let $V$ be a bounded calibrated subaction. Then, there exists $\overline{K}$ such that for any $k\geq\overline{K}$, the restriction $V|_{\Sigma_k}$
 is a calibrated subaction.
\end{prop}
\begin{proof}
Given $V$ a bounded calibrated subaction and $x\in\Sigma$, if   $\sigma(y)=x$ e \[V(y)=V(x)-f(y)+m,\]
then \begin{equation}\label{afm2}
         f(y)=V(x)-V(y)+m\geq \inf V-\sup V +m.
        \end{equation}
On the other hand  $f$ is a  coercive potential, then there exists $J$ such that if $j>J$,
\begin{equation*}
 \sup f|_{[j]}<\inf V-\sup V +m.
\end{equation*}
therefore, any $y$ that satisfies \eqref{afm2} also satisfies $y_0\leq J$. So, if $\Sigma_k$ satisfies
$\{0,1,\dots,J\}\subset\mathcal{A}_k$, then for any  $x\in\Sigma_k $ there  exists $y\in\Sigma$ such that $\sigma(y)=x$ and
$V(y)=V(x)-f(y)+m$.  for the previous observations $y_0\leq J$, then $y\in\Sigma_k$. This proves that $V|_{\Sigma_k}$ is a calibrated subaction.
\end{proof}
For any $k$, $f_k$ is a potential with unique maximizing measure. Then for any  $j>J$, $V|_{\Sigma_j}$ is a calibrated subaction in $\Sigma_j$. By using results in \cite{blt06} we have that for any  $x\in\Sigma_j$ and $\overline{y}\in
supp(\mu_{max})$
\begin{equation}\label{item3}
 V(x)=S_{f_k}(\overline{y},x)+V(\overline{y}).
\end{equation}
This equality is true in the dense set  $\cup_{j>J}\Sigma_j$, by continuity of $V$ and  $S_f$, we can conclude that  for all $x\in\Sigma$, equation \ref{item3} is satisfied.

\endproof

Let us consider $V_1$ and $V_2$ two calibrated bounded  subactions and  $x\in\Sigma$, \eqref{item3} implies
\begin{equation*}
 V_1(x)-V_1(\overline{y})=V_2(x)-V_2(\overline{y}).
\end{equation*}
Then
\begin{equation*}
 V_1(x)-V_2(x)=V_2(\overline{y})-V_1(\overline{y}).
\end{equation*}
In other words, two bounded calibrated subactions differ by a constant.

Theorem B is proved. Let us now prove the last theorem.

\proof[proof of theorem \ref{ThmC}]

By the part (a) of theorem B, it is sufficient to prove theorem C for $V=S_f$.

Let us suppose that $S_f$ is a bounded potential.

Given $x\in\Sigma$ consider  $y\in\Sigma$ such that $\sigma(y)=x$ and $S_f(y)=S_f(x)-f(y)+m$.
On one hand
\begin{equation*}
 f(y)\geq \inf S_f-\sup S_f+m:=K> - \infty.
\end{equation*}
On the other hand, $f$ is a coercive potential. Let $J$ be such that for any $j>J$, $\sup f|_{[j]}<K$. Then for all $x\in\Sigma$ there exists $y$
such that $\sigma(y)=x$ and $y_0<J$, this is, $\Sigma$ satisfies BP property.

Now, let us suppose that $\Sigma$ is a BP shift. Let us consider $I$ such that for any  $x\in\Sigma$ there exist  $i<I$ with $ix\in\Sigma$. Given $i<I$ let us define $x^i\in\Sigma_J$ such that
$ix^i\in\Sigma_L$, $L$ is chosen such that $i\in\mathcal{A}_L$ for any $i<I$. Given $\epsilon>0$, for any
$x\in\Sigma$, there exists  $z$ such that $\sigma^n(z)=x$ and $d(z,\overline{y})<\epsilon$. In addition  $z_{n-1}<I$ and there exists
$\tilde{z}$ such that  $\sigma^l(\tilde{z})=x$ and $\tilde{z}_i<I$ for any $i<l$. Then
\begin{equation*}
 S_f(x)\geq  S_{f_L}(x^{z_{n-1}})-Var(f)> -\infty.
\end{equation*}
So, if  $x\in\Sigma$
\begin{equation*}
 S_f(x)\geq  \inf_{i<I}S_{f_L}(x^{i})-Var(f)> -\infty.
\end{equation*}
Then $S_f$ is a bounded bellow potential, which concludes the proof.

\endproof

If we assume the uniqueness of the maximizing measure, by the last part of theorem \ref{ThmB} we get the corollary, as follows.

\proof
Equivalence of (a) and (b) is theorem C.

We know (a) implies (c), since the barrier is always a calibrated subaction by theorem A.

Finally, by the last part of theorem B (and the uniqueness of the maximizing measure), we have (c) implies (a).

\endproof

\section{Renewal shifts}\label{renewal}
In this section we construct a family of examples where there is not any bounded calibrated subaction. These examples are based on the Renewal shifts as described in \cite{iommi,sarig01} and references therein.

From the corollary, we know that in order to find an example in which no bounded calibrated subaction exists, it is sufficient to consider a non BP transitive Markov shift.

Renewal \textit{shifts} form a class of topologically mixing (hence, topologically transitive) systems that does not satisfy the BIP condition. We use them to construct examples with the  BI property but without the BP property. If the transition matrix is the transpose of the previous one we obtain a BP shift that is not BI.

A Renewal  \textit{shift} is a topologically mixing Markov shift such that for each $n$ there exist at most one periodic orbit  $x=x_0x_1\dots x_{n-1}$ of period $n$ where  $x_j=0$ if, and only if
$j=kn$ for some  $k\in\mathbb{N}$.

Let $A$ be a transition matrix defined by $A(i,j)=1$ if, and only if $i,j$ are in the next cases
\begin{enumerate}
 \item $i=j=0$,
 \item $i=j+1$
 \item $i=0, j=d_n$ for some $n\in\mathbb{N}$.
\end{enumerate}
\begin{figure}
\centering
\begin{tikzpicture}[->,>=stealth',shorten >=1pt,auto,node distance=1cm,
                thick]
  \node (1) {0};
  \node (2) [right of=1] {$1$};
  \node (3) [right of=2] {\dots};
  \node (4) [right of=3] {$d_1$};
  \node (5) [right of=4,node distance=1.7cm] {$(d_1+1)$};
  \node (6) [right of=5,node distance=1.7cm] {$\dots$};
  \node (7) [right of=6] {$d_2$};
  \node (8) [right of=7,node distance=1.3cm] {$\dots$};
  \path
    (1) edge [loop left] (1)
        edge [bend left] (4)
        edge [bend left] (7)
    (2) edge  [below](1)
    (3) 
        edge (2)
    (4) edge (3)
    (5) edge (4)
    (6) edge (5)
    (7) edge (6)
    (8) edge (7);

\end{tikzpicture}
     \caption{Renewal shift}
   \label{renewalfig}
   \end{figure}
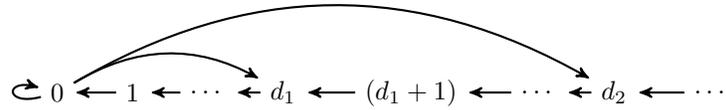

By this construction and recalling the definition of the BP condition in the preliminaries, notice that a Renewal shift satisfies the BP condition if, and only if, there is $J\in\mathbb{N}$ such that $0j$ is an admissible word for all $j \geq J$.

Let, for example, $\Sigma$ be a Renewal shift with $d_i=2i$. As a consequence of theorem \ref{ThmC}, for any coercive potential $f$ with bounded variation defined on $\Sigma$, any bounded calibrated subaction cannot exist.

\section{Acknowledgement}
The first author thanks CAPES for the support within the development of this work.


\end{document}